\documentclass{aims}
\usepackage{amsmath}
\usepackage{paralist}
\usepackage{graphics} 
\usepackage{epsfig} 
\usepackage{graphicx} 
\usepackage{epstopdf} 
\usepackage[colorlinks=true]{hyperref}
\hypersetup{urlcolor=blue, citecolor=red}

\usepackage{graphicx}
\usepackage{amsfonts}
\usepackage{amsmath}
\usepackage{amssymb}
\usepackage{fancyhdr}
\usepackage{titlesec}
\usepackage{indentfirst}
\usepackage{booktabs}
\usepackage{verbatim}
\usepackage{color}
\usepackage{latexsym}
\usepackage{amscd}
\usepackage{esvect}
\usepackage{graphicx}
\usepackage{upgreek}
\usepackage{subcaption}
\usepackage{caption}
\usepackage[page,header]{appendix}
\usepackage{titletoc}
\usepackage{mathrsfs}
\usepackage[numbers]{natbib}
\usepackage{float}
\usepackage{lipsum}
\usepackage[export]{adjustbox}
\usepackage{amsthm}

\usepackage{amsmath}               
  {
      \theoremstyle{plain}
      \newtheorem{assumption}{Assumption}
  }

\allowdisplaybreaks

\def\be{{\bf{e}}}

\def\bl{{\bf{l}}}

\def\bj{{\bf{j}}}

\def\currentvolume{X}
 \def\currentissue{X}
  \def\currentyear{200X}
   \def\currentmonth{XX}
    \def\ppages{X--XX}
     \def\DOI{10.3934/xx.xx.xx.xx}

\newtheorem{theorem}{Theorem}[section]

\newtheorem{lemma}[theorem]{Lemma}

\theoremstyle{definition}

\newtheorem{remark}{Remark}

\title[UQ for Multiscale Semiconductor Boltzmann with Random Inputs] 
      {Uniform Spectral Convergence of the Stochastic Galerkin method for the Linear Semiconductor Boltzmann Equation with Random Inputs and Diffusive Scaling}

\author[Liu Liu]{}

\thanks{
The author was supported in part by Prof. Shi Jin's NSF grants DMS-1522184 and DMS-1107291: RNMS KI-Net.}

\begin{document}
\maketitle

\centerline{\scshape Liu Liu $^*$}
\medskip
{\footnotesize
 \centerline{Department of Mathematics, University of Wisconsin-Madison,}
  \centerline{Madison, Wisconsin, 53706, USA}
 } 

\bigskip

\begin{abstract}
In this paper, we study the generalized polynomial chaos (gPC) based stochastic Galerkin method for the linear
semiconductor Boltzmann equation under diffusive scaling and with random inputs from an {\it anisotropic} collision kernel and the random initial condition. 
While the numerical scheme and the proof of {\it uniform-in-Knudsen-number regularity} of the distribution function in the random space
has been introduced in 
\cite{JinLiu-2016}, the main goal of this paper is to first obtain a sharper estimate on the regularity of the solution--an exponential decay towards its local equilibrium, 
which then lead to the {\it uniform spectral convergence} of the stochastic Galerkin method for the 
problem under study. 
\end{abstract}

\section{Introduction}

Despite the vast amount of existing research and the ever-growing trend of development of kinetic theory \cite{Cer},  
the study of kinetic equations
remained deterministic and ignored uncertainty in the model, which might yield inaccurate solution for practical problems, for example,  
in mesoscopic modeling of physical, biological and social sciences. 

In this paper, we consider the linear semiconductor Boltzmann equation \cite{MRS,Jungel} with random inputs, which arise from the collision kernel and initial data. 
There are many sources of uncertainties that can arise in kinetic equations. 
For example, the collision or scattering kernel contained
in the integral operator that models the interaction mechanism between particles 
should be calculated from first principles ideally, 
which is extremely complicated for complex particle systems. Thus empirical collision kernels are usually used in 
practice and measurement errors may arise \cite{HBS}.
Other sources of uncertainties can be due to inaccurate measurement of the initial and boundary data, forcing or source terms, gas-surface interactions and
geometry. 
The uncertainties are not limited to the above examples.
Understanding the impact and propagation of uncertainties is essential to the simulations and validation of the complex kinetic systems, and furthermore,
provides reliable predictions and better risk assessment for scientists and engineers. 

We apply the generalized polynomial chaos approach in the stochastic
Galerkin (referred as gPC-SG) framework \cite{GS, LMK, XiuBook}.  
Compared with the classical Monte-Carlo method, the gPC-SG approach enjoys a 
spectral accuracy in the random space--if the solution is sufficient regular--while the Monte-Carlo method converges with only half-th order accuracy. 
For recent activities for uncertainty quantificaiton in kinetic theory,
we refer to a recent review article \cite{HuReview}, which surveyed recent results in the study of kinetic equations with random inputs,
\cite{JinLiu-2016, Mu-Liu, JinHu-2016, ShuHuJin, JinLiuMa-2016, QinWang, JinZhu, LuJin, Erin, ShuJin}, 
including their mathematical properties such as regularity and long-time behavior in the random space, 
construction of efficient stochastic Galerkin methods and handling of multiple scales by s-AP schemes. 

Recently, the authors in \cite{LiuJin-2017} have provided a general framework using hypocoercivity to study general class of linear and nonlinear 
kinetic equations with uncertainties from the initial data or collision kernels in both incompressible Navier-Stokes and Euler
(acoustic) regimes. For initial data near the global Maxwellian, an exponential convergence
of the random solution toward the deterministic global equilibrium, a spectral accuracy and exponential decay of the numerical error of
the SG system have been established. 
We also mention some recent work on uncertainty quantification for hyperbolic equations \cite{DGX, Jin-Ma, Tao-Zhou1, Tao-Zhou2}
and highly oscillatory transport equations arisen from 
{\it non-adiabatic} transition in quantum dynamics \cite{LemouJinLiu}. \\[1pt]

Consider the linear transport equation with {\it anisotropic} collision operator. 
Let $f(t,x,v,z)$ be the probability density distribution for particles at position $x\in\mathbb R^d$, 
with velocity $v\in\mathbb R^d$ for $t\geq 0$. 
 $f$ solves the following kinetic equation with random inputs, 
\begin{align}
&\label{eqn}\displaystyle \epsilon \partial_t f + v\cdot\nabla_x f = \frac{1}{\epsilon}\mathcal Q(f),  \\[6pt]
&\displaystyle \mathcal Q(f) = \int_{\mathbb R^d} \sigma(v,w,z) \left(M(v)f(w,z) - M(w)f(v,z)\right) dw, 
\end{align}
where $$M(v) = \frac{1}{(2\pi)^{d/2}}e^{-\frac{|v|^2}{2}}\,,$$
is the normalized Maxwellian distribution of the electrons. 
$\epsilon$ is the Knudsen number, which measures the ratio between the particle mean free path and a typical length scale. 
The anisotropic collision operator $\mathcal Q$ describes a linear approximation of the electron-phonon interaction. It is bounded and nonnegative on a 
suitable Hilbert space (\cite{Pou}) and has a one-dimensional kernel spanned by $M$. 

We assume the anisotropic scattering kernel $\sigma$ to be symmetric and bounded, 
\begin{equation}\label{sigma}\sigma_{\rm min} \leq \sigma(v,w,z)=\sigma(w,v,z) \leq \sigma_{\text{max}}. \end{equation}
$\sigma$ can be random in reality and
one assumes that it depends on the random 
variable $z\in\mathbb R^d$, with support $I_{z}$ and a prescribed probability density function $\pi(z)>0$. 

Denote \begin{equation}\label{Pi_Notation}\Pi f= M(v)\int_{\mathbb R^d} f dv, \end{equation} by the local equilibrium function. 
A periodic boundary condition in space is assumed. 
The initial condition can be random and is given by $$f|_{t=0}(x,v,z) = f_0(x,v,z).$$

One challenge in numerical approximations of kinetic and transport equations is the 
varying magnitude of the Knudsen number. 
Kinetic equations for highly integrated semiconductor devices have a diffusive scaling, measured by $\epsilon$. 
When $\epsilon$ goes to zero, the high scattering rate of particles leads the transport equation to a diffusion equation (\ref{diff1}),
known as the diffusion limit \cite{MRS, Pou, BSS}. 

For each value of random variable $z$, (\ref{eqn}) is a deterministic equation. 
As $\epsilon\to 0$, ${\mathcal Q}(f)=0$, since $\text{Kern}\mathcal Q=\text{Span}\{M(v)\}$, then $\displaystyle f(t,x,v,z)=\rho(t,x,z)M(v)$, where $\rho$ satisfies a random drift-diffusion equation
\cite{DJ,Pou,MRS,Klar}:
\begin{equation}   
\partial_{t}\rho=\nabla_{x}\cdot(D \nabla_{x}\rho), 
\label{diff1} 
\end{equation}
where the diffusion matrix $D$ is defined by $\displaystyle D=\int_{\mathbb R^d}\frac{v\otimes vM(v)}{\lambda(v,z)}dv$. 
The limit $\epsilon\to 0$ is known as the drift-diffusion limit. 

When $\epsilon$ is small, the equation becomes numerically stiff and requires expensive computational cost. 
To overcome this difficulty, asymptotic-preserving (AP) schemes \cite{Jin-Review, Jin-99, GJL, HuReview} are designed to mimic the asymptotic transition from the kinetic equations to the hydrodynamic limit, in the discrete setting \cite{LemouMieussens2008, JinLorenzo, Klar, JPT2}. 
The scheme automatically becomes a consistent discretization of the limiting macroscopic equations as $\epsilon\to 0$. 
The idea of stochastic asymptotic-preserving (s-AP) schemes was recently introduced in \cite{XZJ} for random kinetic equations with multiple scales. 
s-AP schemes in the gPC-SG framework allows the use of mesh sizes, 
time steps and the number of terms in the orthogonal polynomial expansions independent of the Knudsen number. 
The solution approaches, as $\epsilon\to 0$, to the gPC-SG method for the corresponding limiting, macroscopic equation with random inputs. 

In \cite{JinLiuMa-2016}, the authors prove the uniform regularity of the linear transport equations with random {\it isotropic} scattering coefficients, 
random initial data and diffusive scaling. 
\cite{QinWang} also carries out the analysis in a general setting and 
proves, using hypocoercivity for linear collision operators, the uniform regularity in the random space for all linear kinetic equations that conserve mass with random inputs. Their results hold true in kinetic, parabolic and high field regimes. 
Moreover, with an estimate on the regularity of $f-\Pi f$ in the random space, 
the authors in \cite{JinLiuMa-2016} are able to prove the {\it uniform} spectral convergence of the 
stochastic Galerkin method, a result that both \cite{JinLiu-2016} and \cite{QinWang} do not have. 
The main goal of this paper is to extend the results of \cite{JinLiuMa-2016} to the case of {\it anisotropic} scattering. Namely, we
first obtain the {\it uniform regularity} in the random space, then prove the 
{\it uniform spectral convergence} of the stochastic Galerkin method for problem (\ref{eqn}). 

Although the idea of the proof follows the line as in \cite{JinLiuMa-2016}, there are several differences due to the anisotropy of the collision kernel.
First, the specific estimates in all estimates are  different when one treats the anisotropic scattering. 
The major difference lies in the proof of the exponential decay of $f-\Pi f$. 
In contrast to the bounded velocity $v\in[-1,1]$ in \cite{JinLiuMa-2016}, $v\in\mathbb R^d$ in our problem, 
and an exponential decay estimate for $v\cdot\nabla_x f$ is needed, which brings up the main difficulty. 

In \cite{JinLiu-2016}, the uniform regularity was proved in the random space for problem (\ref{eqn}) by using the symmetric 
property of the collision operator $\mathcal Q$. To carefully specify the constant coefficients in the proof, 
linear dependence on $z$ of the collision kernel was assumed in \cite{JinLiu-2016}. In this paper our analysis does
not require the linearity in $z$.

This paper is organized as follows. 
We first introduce the gPC-SG method in section \ref{gPC}. 
Estimates on the regularity of the distribution function $f$ in the random space are studied in section \ref{RE}: 
subsection \ref{UR} proves the uniform regularity of $f$; subsection \ref{vf} gives an estimate of the regularity of $v\cdot\nabla_x f$, 
which serves as a building block to obtain the exponential decay of the regularity of 
$\Pi f-f$, a result shown in subsection \ref{Pi_f}. 
With all the results obtained in section \ref{RE}, we prove the uniform convergence 
of the gPC-SG method for problem (\ref{eqn}) in section \ref{UC}. Lastly, conclusion is provided in section \ref{Con}. 

\section{The gPC Stochastic Galerkin Approximation}
\label{gPC}
We briefly review the gPC method and its Galerkin formulation. 
In the gPC setting, one seeks for a numerical solution in term of $d-$
variate orthogonal polynomials of degree $N\geq 1$. The linear space $V_z$ is set to be $\mathbb P_N^d$, the space of 
$d$-variate orthonormal polynomials of degree up to $N\geq 1$. For random variable 
$z\in I_{z}\subset\mathbb R^d$, one approximates the solution $f$ by an orthogonal polynomial expansion $f_K$, that is, 
$$f(t,x,v,z)\approx f_K(t,x,v,z)=\sum_{|\bj|=1}^K\alpha_\bj(t,x,v)\psi_\bj(z), \,\, K= \binom{d+N}{d}, $$
where $\bj=(j_1,\dots, j_d)$ is a multi-index with $|\bj|=j_1+\dots+j_d$. 
$\{\psi_{\bj}(z)\}$ are the orthonormal basis functions that form $\mathbb P_N^d$ and satisfy
\begin{equation}
\label{Kronecker}
\int_{I_{z}} \psi_\bj(z)\psi_\bl(z)\pi(z)dz=\delta_{\bj\bl}, \qquad 1\leq |\bj|, |\bl|\leq K=\text{dim}(\mathbb P_N^d), 
\end{equation}
where $\pi(z)$ is the probability density function of $z$ and $\delta_{\bj\bl}$ the Kronecker Delta function. 

The orthogonality with respect to $\pi(z)$ defines the orthogonal polynomials. For example, 
the Gaussian distribution defines the Hermite polynomials; the uniform distribution defines the
Legendre polynomials; and the Gamma distribution defines the Laguerre polynomials, etc.
If the random dimension $d>1$, one can re-order the multi-dimensional polynomials $\{\psi_{\bj}(z), \, 1\leq |\bj|\leq K\}$
into a single index $\{\psi_k(z), \,1\leq k\leq N_k=\text{dim}(\mathbb P_N^d)=\binom{d+N}{d} \}$. 

By the gPC-SG approach, applying the ansatz $$\rho(t,x,z)=\sum_{|\bj|=1}^K \hat\rho_{\bj}(t,x)\psi_{\bj}(z)$$ and conducting the Galerkin projection 
of limiting diffusion equation (\ref{diff1}), one obtains a gPC approximation of the random diffusion equation (\cite{JinLiu-2016})
\begin{equation}\partial_t\hat\rho_{\bj}=\nabla_x\cdot(T\sum_{\bl}S_{\bj\bl}\nabla_x\hat\rho_{\bl}),  \label{diff2} \end{equation}
where $\displaystyle T=\int_{\mathbb R^d}v\otimes v M(v)dv$ and 
the gPC coefficient matrix $S=(S_{\bj\bl})_{K\times K}$ 
is given by \begin{align*} S_{\bj\bl}=\int_{I_{z}}\frac{1}{\lambda(v,z)}\psi_{\bj}(z)\psi_{\bl}(z)\pi(z)dz. \end{align*}

It has been demonstrated in \cite{JinLiu-2016} that solutions of the gPC-SG scheme converge spectrally to that of the
Galerkin system of the diffusion equation given by (\ref{diff2}). One can refer to Theorem 2.2 and Theorem 4.3 in \cite{JinLiu-2016} on the uniform regularity and a spectral accuracy (not uniform) of the gPC-SG method. 
The goal of this paper is to give a theoretical proof of 
the {\it uniform spectral convergence} of the SG method with respect to $\epsilon$.  

\section{Regularity Estimates}
\label{RE}
\subsection{Notations}
We first introduce the Hilbert space of the random variable, 
$$H(\mathbb R^d; \pi(z)dz) = \left\{  f \,|\,\mathbb R^d\to \mathbb R, \, \int_{\mathbb R^d}f^2(z)\pi(z)dz<\infty\right\}, $$
equipped with the corresponding inner product and norm
$$\langle f, g \rangle_{\pi} = \int_{I_z} f g \pi(z)\, dz, \qquad ||f||_{\pi}^2 = \langle f, f\rangle_{\pi}. $$
Define the $k$-th order differential operator with respect to $z$ as 
$$D^{k} f(t,x,v,z) := \partial_z^k f(t,x,v,z), $$
and the Sobolev norm in $H$ as $$ ||f(t,x,v,\cdot)||_{H^k}^2 := \sum_{\alpha\leq k}||D^{\alpha}f (t,x,v,\cdot) ||_{\pi}^2\,.$$

We introduce the Hilbert space of the velocity variable 
$\displaystyle\widetilde L^2 := L^2\left(\mathbb R^d; \frac{dv}{M(v)}\right)$, with the corresponding inner product
$\langle \cdot , \cdot \rangle_{\widetilde L^2}$ and norm $||\cdot||_{\widetilde L^2}$. 
By the coercivity property of the collision operator $\mathcal Q$ (\cite{RSZ}), for any $f\in\widetilde L^2$, we have
\begin{equation}\label{coer}\langle \mathcal Q(f), f\rangle_{\widetilde L^2} \leq -\sigma_{\rm min} || f - \Pi f ||_{\widetilde L^2}^2\,,
\end{equation}
where $\Pi f$ is the orthogonal projection of $\widetilde L^2$ onto $\text{Kern}\mathcal Q$ and is given in (\ref{Pi_Notation}). 
Let $x\in\Omega$, $v\in\mathbb R^d$, $z\in I_z$.  
Introduce the energy norms
\begin{align*}
&\displaystyle ||f(t,\cdot, \cdot, \cdot)||_{\Gamma}^2 :=  \int_{\Omega}\int_{\mathbb R^d} \frac{||f||_{\pi}^2}{M(v)}\,  dvdx,  \\[6pt]
&\displaystyle ||f(t,\cdot, \cdot, \cdot)||_{\Gamma^k}^2 := \int_{\Omega}\int_{\mathbb R^d}  \frac{||f||_{H^k}^2}{M(v)}\, dvdx. 
\end{align*}
For simplicity, we will suppress the $t$ dependence and denote $||f||_{\Gamma}$, $||f||_{\Gamma^k}$ in the following. 

\subsection{Regularity in the Random Space}
\label{UR}

In this section, we prove that the solution $f$ will preserve the regularity of the initial data in the random space. 
For simplicity, the following lemmas and theorems are stated only for one-dimensional case. 
Proof for the high dimensional case is identical except for the change of coefficients. 

We first show Lemma \ref{L1}, which will help us get the uniform regularity of $f$, a result given in Theorem \ref{thm1}. 

\begin{lemma}
\label{L1}
For any $k\geq 0$, there exist $k$ constants $C_{kj}>0$, $j=0, \cdots, k-1$ such that 
\begin{align}\label{MI}\displaystyle \epsilon^2 \partial_t \bigg( ||D^k f||_{\Gamma}^2 + \sum_{j=0}^{k-1}C_{kj} ||D^j f||_{\Gamma}^2 \bigg) \leq 
\begin{cases}   
  -2\sigma_{\rm min}\, ||\Pi f-f||_{\Gamma}^2\,,\qquad k=0, \\[6pt]
  -\sigma_{\rm min}\, ||D^k (\Pi f- f)||_{\Gamma}^2\,,\quad k\geq 1. 
\end{cases}
 \end{align}
 \end{lemma}
\begin{proof}
The idea of the proof is similar to that in \cite{JinLiuMa-2016}. However, there are some differences due to the anisotropic collision operator. 
We will prove this Lemma by using Mathematical Induction. 

When $k=0$, (\ref{MI}) holds because of the coercivity property given by (\ref{coer}). 

Assume that (\ref{MI}) holds for any $k\leq p$, where $p\in \mathbb N$. Adding all these inequalities, we get
$$ \epsilon^2 \partial_t \bigg(\frac{1}{2}||f||_{\Gamma}^2 + \sum_{i=1}^p ||D^i f||_{\Gamma}^2 + \sum_{i=1}^p \sum_{j=0}^{i-1}C_{ij} || D^j f||_{\Gamma}^2 
\bigg) \leq -\sigma_{\rm min}\, || \Pi f -f||_{\Gamma^p}^2\,,
$$
which is equivalent to 
\begin{equation}\label{p0}\epsilon^2 \partial_t \bigg(\sum_{j=0}^p C_{p+1,j}^{\prime} ||D^j f||_{\Gamma}^2 \bigg) \leq -\sigma_{\rm min}\, || \Pi f -f||_{\Gamma^p}^2\,,
\end{equation}
where 
\begin{align*}
\displaystyle
C_{p+1,j}^{\prime} =
 \begin{cases} \frac{1}{2} + \sum_{i=1}^p C_{i 0}, \qquad j=0, \\[4pt]
                        1 + \sum_{i=1}^p C_{i j}, \qquad 1\leq j\leq p-1, \\[4pt]       
                        1, \qquad\qquad\qquad\quad j=p.                 
 \end{cases}
\end{align*}

For $k\geq 1$, take $k$-th order formal differentiation of (\ref{eqn}) with respect to $z$, 
\begin{equation}\label{Dk}
\epsilon^2 \partial_t (D^k f) + \epsilon v\cdot \nabla_x (D^k f) =D^k \mathcal Q(f). \end{equation}
Denote $d\mu=dx\pi(z)dz$, and $\mathcal S=\Omega\times I_z$. 
Taking a scalar product with $D^k f$, dividing by $M(v)$ to both sides of (\ref{Dk}) and integrating on $\Omega \times \mathbb R^d \times I_z$, one has
\begin{align}
&\quad\displaystyle\label{p+1}\frac{\epsilon^2}{2} \partial_t ||D^k f||_{\Gamma}^2 + \epsilon \int_{\mathcal S}\, \langle v\cdot \nabla_x (D^k f), D^k f\rangle_{\widetilde L^2}\,  d\mu=\int_{\mathcal S}\,\langle D^k \mathcal Q(f), D^k f \rangle_{\widetilde L^2}\, d\mu\,.
\end{align}
By the periodic boundary condition in space, 
$$\int_{\mathcal S}\, \langle v\cdot \nabla_x (D^k f), D^k f\rangle_{\widetilde L^2}\, d\mu=0\,. $$
Note that right-hand-side of (\ref{p+1}) is given by 
\begin{align}
&\displaystyle\int_{\mathcal S}\, \langle D^k \mathcal Q(f), D^k f \rangle_{\widetilde L^2}\, d\mu =
\int_{\mathcal S}\, \sum_{j=0}^{k-1}\binom{k}{j} \langle \widetilde{\mathcal Q}_{kj}(D^j f), D^k f\rangle_{\widetilde L^2}\, d\mu \notag\\[2pt]
&\label{RHS}\displaystyle \qquad\qquad\qquad\qquad\qquad\qquad +\int_{\mathcal S}\,\langle \mathcal Q(D^k f), D^k f\rangle_{\widetilde L^2}\, d\mu, 
\end{align}
where we define $$\widetilde{\mathcal Q}_{kj}(f) = \int_{\mathbb R^d} D^{k-j} \sigma(v,w,z) \left(D^j f(w,z)M(v) - D^j f(v,z)M(w) \right) dw.$$

By coercivity given in (\ref{coer}), the second term in (\ref{RHS}) satisfies
\begin{equation}\label{second}\int_{\mathcal S}\,\langle \mathcal Q(D^k f), D^k f\rangle_{\widetilde L^2}\, d\mu \leq -\sigma_{\rm min}\, || D^k (f- \Pi f)||_{\Gamma}^2\,.
\end{equation}
Notice that
$$\int_{\mathcal S}\,\langle \widetilde{\mathcal Q}_{kj}(D^j f), D^k(\Pi f)\rangle_{\widetilde L^2}\, d\mu 
= \int_{\mathcal S}\,\int_{\mathbb R^d}  \widetilde{\mathcal Q}_{kj}(D^j f) dv \int_{\mathbb R^d}D^k f(w)dw\, d\mu =0\,. $$
By Young's inequality, the first term in (\ref{RHS}) satisfies the estimate: 
\begin{align}
&\displaystyle \quad\int_{\mathcal S}\,\sum_{j=0}^{k-1}\binom{k}{j} \langle \widetilde{\mathcal Q}_{kj}(D^j f), D^k f\rangle_{\widetilde L^2}\, d\mu = 
\int_{\mathcal S}\,\sum_{j=0}^{k-1}\binom{k}{j} \langle \widetilde{\mathcal Q}_{kj}(D^j f), D^k(f-\Pi f)\rangle_{\widetilde L^2}\, d\mu \notag \\[4pt]
&\label{first}\displaystyle \leq 
\frac{\sigma_{\rm min}}{2}\,||D^k (f-\Pi f)||_{\Gamma}^2 + \frac{1}{2\sigma_{\rm min}} \, \bigg|\bigg|\, \sum_{j=0}^{k-1}\binom{k}{j} 
\widetilde{\mathcal Q}_{kj}(D^j f)\, \bigg|\bigg|_{\Gamma}^2\,,
\end{align}
then by using the Cauchy-Schwartz inequality, one has
\begin{align}
&\displaystyle\quad  \bigg|\bigg|\, \sum_{j=0}^{k-1}\binom{k}{j} 
\widetilde{\mathcal Q}_{kj}(D^j f)\, \bigg|\bigg|_{\Gamma}^2 \\[4pt]
&\displaystyle \leq \sum_{j=0}^{k-1} \binom{k}{j}^2\, 
\sum_{j=0}^{k-1}\, \int_{\mathcal S}\,  \bigg|\bigg|\int_{\mathbb R^d}D^{k-j}\sigma\, (M(v)D^j f(w,z) - M(w)D^j f(v,z)) dw\, \bigg|\bigg|_{\widetilde L^2}^2\, d\mu\notag\\[4pt]
&\displaystyle \leq 4^k \left(\max_{0\leq j\leq k} ||D^j\sigma||_{L^{\infty}}^2\right) \left(\sum_{j=0}^{k-1}\,\int_{\mathcal S}\, ||D^j (\Pi f - f)||_{\widetilde L^2}^2\, d\mu\right)\notag\\[4pt]
&\label{first1}\displaystyle  = 4^k \left(\max_{0\leq j\leq k} ||D^j\sigma||_{L^{\infty}}^2\right)\, ||\Pi f -f ||_{\Gamma^{k-1}}^2\,.
\end{align}

Therefore, by (\ref{second}), (\ref{first}) and (\ref{first1}), one has
\begin{align}
&\displaystyle\int_{\mathcal S}\,\langle D^k \mathcal Q(f), D^k f \rangle_{\widetilde L^2}\, d\mu \leq -\frac{\sigma_{\rm min}}{2} ||D^k(f-\Pi f)||_{\Gamma}^2 \notag\\[2pt]
&\label{RHS1}\displaystyle\qquad\qquad\qquad\qquad\qquad\qquad +\frac{4^k}{2\sigma_{\rm min}}
\left(\max_{0\leq j\leq k} ||D^j\sigma||_{L^{\infty}}^2\right)\, ||\Pi f -f ||_{\Gamma^{k-1}}^2\,.
\end{align} 
\\[2pt]

When $k=p+1$, (\ref{p+1}) and (\ref{RHS1}) read
\begin{align}
&\displaystyle \epsilon^2 \partial_t ||D^{p+1} f||_{\Gamma}^2 \leq -\sigma_{\rm min}\, ||D^{p+1}(f-\Pi f)||_{\Gamma}^2 + \frac{4^{p+1}}{\sigma_{\rm min}}
\left(\max_{0\leq j\leq k} ||D^j\sigma||_{L^{\infty}}^2\right)\, ||\Pi f -f ||_{\Gamma^{p}}^2 \notag\\[4pt]
&\label{p+1 term}\displaystyle \qquad\qquad\qquad \leq -\sigma_{\rm min} ||D^{p+1}(f-\Pi f)||_{\Gamma}^2 + \frac{4^{p+1} C_{\sigma}^2}{\sigma_{\rm min}}\, 
||\Pi f -f ||_{\Gamma^{p}}^2\,.
\end{align}
Multiplying (\ref{p0}) by $4^{p+1}C_{\sigma}^2/\sigma_{\rm min}^2$ and adding to (\ref{p+1 term}), one has
$$\epsilon^2 \partial_t \bigg(||D^{p+1} f||_{\Gamma}^2 + \sum_{j=0}^p C_{p+1,j}
||D^j f||_{\Gamma}^2 \bigg) \leq -\sigma_{\rm min}\, ||D^{p+1}(f-\Pi f)||_{\Gamma}^2\,,$$ 
where $\displaystyle C_{p+1,j}=\frac{4^{p+1}C_{\sigma}^2}{\sigma_{\rm min}^2}\, C_{p+1,j}^{\prime}\,.$
This shows that (\ref{MI}) holds for $k=p+1$. 
By Mathematical Induction, (\ref{MI}) holds for all $k\in\mathbb N$. Thus we finish the proof of Lemma \ref{L1}.  
\end{proof}
Theorem \ref{thm1} below shows that the solution $f$ will preserve the regularity of the initial data in the random space at later time, 
in the energy norm $\Gamma$. 

\begin{theorem} {\bf({Uniform\, Regularity})}
\label{thm1}
 Assume $$\sigma(v,w,z) \geq \sigma_{\rm min} >0. $$
 For some integer $m\geq 0$, $$||D^{k}\sigma||_{L^{\infty}(v,z)} \leq C_{\sigma}, \qquad\quad
 || D^{k} f_0 ||_{\Gamma} \leq C_0, \qquad k=0, \cdots, m, $$
 then the solution $f$ to (\ref{eqn}) satisfies $$|| D^{k} f ||_{\Gamma} \leq C, \qquad k=0, \cdots, m, \qquad \forall t>0, $$
 where $C_{\sigma}$, $C_0$ and $C$ are constants independent of $\epsilon$. 
 \end{theorem}
 \begin{proof}
According to Lemma \ref{L1}, one has
$$ \partial_t \bigg( ||D^k f||_{\Gamma}^2 + \sum_{j=0}^{k-1} C_{kj}||D^j f||_{\Gamma}^2 \bigg) \leq 0, \qquad C_{kj}>0, \, k\in\mathbb N, $$
which gives 
\begin{align*}
&\displaystyle ||D^k f||_{\Gamma}^2 \leq ||D^k f||_{\Gamma}^2 + \sum_{j=0}^{k-1} C_{kj}||D^j f||_{\Gamma}^2 \\[4pt]
&\displaystyle\qquad\qquad\leq||D^k f_0||_{\Gamma}^2 + \sum_{j=0}^{k-1} C_{kj}||D^j f_0||_{\Gamma}^2 \leq C_0^2 \,(1+\sum_{j=0}^{k-1} C_{kj}):= C^2\,,
\end{align*}
where $C$ is independent of $\epsilon$. This completes the proof of the theorem. 
\end{proof}
\begin{remark}
If we consider the linear semiconductor Boltzmann equation with random inputs and external electric potential
 \begin{align}
&\displaystyle \epsilon \partial_t f + v\cdot\nabla_x f +\nabla_x\phi\cdot\nabla_v f  = \frac{1}{\epsilon}\mathcal Q(f), \notag \\[4pt]
&\label{semi}\displaystyle \mathcal Q(f) = \int_{\mathbb R^d} \sigma(v,w,z) \left(M(v)f(w,z) - M(w)f(v,z)\right) dw, 
\end{align}
where the electric potential $\phi=\phi(t,x)$ is given a priori and does not depend on $z$. 
By simply changing $d\mu$ in the proof above to $\displaystyle d\mu=\frac{e^{-\phi}}{M(v)}\, dvdx\pi(z)dz$, one can reach the same result as 
Theorem \ref{thm1}--the uniform regularity of $f$ in the random space. 
However, proving the {\it uniform convergence} of the stochastic Galerkin method for (\ref{semi}) is more complicated and remains a further investigation. 
\end{remark}
 
\subsection{Regularity of $v\cdot\nabla_x f$}
\label{vf}

Differed from the proof in \cite{JinLiuMa-2016} of the estimate on $\Pi f-f$, 
we need to overcome the difficulty to get the regularity of 
$v\cdot\nabla_x f$ in the random space, which is of exponential decay. 
In particular, in the proof of Lemma $4.2$ in \cite{JinLiuMa-2016}, thanks to the
boundedness of $v$, one directly gets $||D^k(v\cdot\nabla_x f)||_{\Gamma} \leq ||D^k (\nabla_x f)||_{\Gamma}$. 

Nevertheless, $v\in\mathbb R^d$ in our problem under study, the above inequality is no longer valid, 
thus a new estimate for $|| D^k (v\cdot\nabla_x f)||_{\Gamma}$ is needed. 
This is the main purpose of the current subsection. 

Firstly, one needs the following assumptions for the collision kernel $\sigma$:
\begin{assumption}
\begin{equation}
\label{assump1}\int_{\mathbb R^d}\int_{\mathbb R^d}\, (D^j \sigma)^2 v^2 M(v)M(w)dwdv\leq \widehat C_{\sigma}^2, \qquad\text{for  }\,  0\leq j\leq k\,,
\end{equation}
\end{assumption}

\begin{assumption}
\begin{equation}\label{assump3} \left|\int_{\mathbb R^d}\, (D^j \sigma) M(w)dw\right|\leq \lambda_1, \qquad\text{for  }\,  1\leq j\leq k\,,
\end{equation}
\end{assumption}
Here $\widehat C_{\sigma}$ and $\lambda_1$ are positive constants. 
Note that when $j=0$ in Assumption $1$, the exact same assumption is used in \cite{RSZ} for the deterministic problem. 
Since $v\cdot\nabla_x f$ and $vf$ satisfy the same equation, $||D^k (vf)||_{\Gamma}$ is estimated for notational simplicity. 
We first prove Lemma \ref{L2}, which will serve as a tool to obtain the main result of this subsection given by Theorem \ref{thm3}. 

\begin{lemma}
\label{L2}
There exist $k$ constants $c_{kj}>0$ for $j=0, \cdots, k-1$ 
and $k+1$ constants $s_{kj}>0$ for $j=0, \cdots, k$, such that 
\begin{align}
&\displaystyle\quad\epsilon^2 \partial_t \bigg(||D^k (vf)||_{\Gamma}^2 + \sum_{j=0}^{k-1} c_{kj}||D^j (vf)||_{\Gamma}^2 \bigg) \notag\\[4pt]
&\label{MI2}\displaystyle\leq 
\begin{cases} \displaystyle -2\sigma_{\rm min}\, ||vf||_{\Gamma}^2 + 2\widehat C_{\sigma}C\, ||vf||_{\Gamma}\,,\qquad\qquad\qquad\quad k=0, \\[6pt]
                    \displaystyle -\sigma_{\rm min} \sum_{j=0}^k ||D^j (vf)||_{\Gamma}^2 + \sum_{j=0}^{k} s_{kj}||D^j (vf)||_{\Gamma}, \qquad  k\geq 1. 
                    \end{cases}
                \end{align}
\end{lemma}
                
\begin{proof}
We prove it by using Mathematical Induction. We first prove the result for $k=0$. 
Multiply by $v$ to both sides of (\ref{eqn}), 
\begin{equation}\label{veqn}\epsilon^2 \partial_t (vf) + \epsilon v\cdot \nabla_x (vf) = v\mathcal Q(f).\end{equation}
One multiplies by $v f$, divides by $M(v)$ to both sides of (\ref{veqn}) and integrates on $\Omega\times\mathbb R^d\times I_z$, then 
\begin{equation}\label{T}\frac{\epsilon^2}{2}\partial_t || vf||_{\Gamma}^2 = \int_{\mathcal S}\int_{\mathbb R^d}\frac{v^2\mathcal Q(f) f(v,z)}{M(v)}\, dvd\mu.
\end{equation}
Denote $\mathcal T=\mathbb R^d\times \mathbb R^d$. The right-hand-side of (\ref{T}) is given by 
\begin{align*}
&\displaystyle\text{RHS}=\underbrace{\int_{\mathcal S}\int_{\mathcal T}v^2 \sigma f(v,z)f(w,z)\, dwdvd\mu}_{\textcircled{a}}
\underbrace{-\int_{\mathcal S}\int_{\mathcal T}\frac{v^2}{M(v)}\sigma M(w)f^2(v,z)\, dwdvd\mu}_{\textcircled{b}}
\end{align*}
By the Cauchy-Schwartz inequality, 
\begin{align}
&\displaystyle \textcircled{a}=\int_{\mathcal S}\int_{\mathcal T}v\sigma\sqrt{M(v)}\sqrt{M(w)}\,
\frac{vf(v,z)}{\sqrt{M(v)}}\frac{f(w,z)}{\sqrt{M(w)}}\, dwdvd\mu \notag\\[4pt]
&\displaystyle \quad\leq \int_{\mathcal S}\left(\int_{\mathcal T} v^2\sigma^2 M(v)M(w)\, dwdv\right)^{\frac{1}{2}}\,
\left(\int_{\mathcal T}\frac{v^2 f^2(v,z)}{M(v)}\, \frac{f^2(w,z)}{M(w)}\, dwdv\right)^{\frac{1}{2}} d\mu \notag\\[4pt]
&\displaystyle \quad \leq \left|\left|\, \bigg(\int_{\mathcal T} v^2\sigma^2 M(v)M(w)\, dwdv \bigg)^{\frac{1}{2}}\, \right|\right|_{L^{\infty}}\, \int_{\mathcal S} ||vf||_{\widetilde L^2}\, 
||f||_{\widetilde L^2}\, d\mu \notag\\[4pt]
&\label{aa}\displaystyle  \quad\leq \widehat C_{\sigma}\, \left(\int_{\mathcal S}||vf||_{\widetilde L^2} ^2 \, d\mu \right)^{\frac{1}{2}}\, 
\left(\int_{\mathcal S}||f||_{\widetilde L^2} ^2 \, d\mu \right)^{\frac{1}{2}}
=\widehat C_{\sigma}\, ||vf||_{\Gamma}\, ||f||_{\Gamma}\,.
\end{align}
Also, 
\begin{align}
\label{bb}
&\displaystyle \textcircled{b} = -\int_{\mathcal S}\int_{\mathbb R^d}\left(\int_{\mathbb R^d}\sigma M(w)dw\right) \frac{v^2 f^2(v,z)}{M(v)}\, dvd\mu \leq -\sigma_{\rm min} \, ||vf||_{\Gamma}^2\,.
\end{align}
Combining (\ref{aa}) and (\ref{bb}), one gets 
\begin{equation}\label{vf0} 
\epsilon^2 \partial_t ||vf||_{\Gamma} \leq  - \sigma_{\rm min}\, ||vf||_{\Gamma}+ \widehat C_{\sigma}||f||_{\Gamma}\,.
\end{equation}
According to Theorem \ref{thm1}, $||f||_{\Gamma}$ is uniformly bounded. 
By Gronwall's inequality, one then has
$$||vf||_{\Gamma} \leq e^{-\frac{\sigma_{\rm min} t}{\epsilon^2}}\, ||vf_0||_{\Gamma}+C^{\prime}\,.$$
\\[2pt]

We now look at the case where $k\geq 1$. Take $D^k$ on both sides of (\ref{veqn}), 
$$\epsilon^2 (D^k vf) + v\cdot\nabla_x (D^k vf) = v D^k \mathcal Q(f), $$
Taking a scalar product with $D^k (vf)$, dividing by $M(v)$ and integrating on $\Omega\times\mathbb R^d\times I_z$, one has
\begin{equation}
\label{veqn2}\frac{\epsilon^2}{2}||D^k vf||_{\Gamma}^2= \underbrace{\int_{\mathcal S}\int_{\mathbb R^d}\frac{v^2 D^k \mathcal Q(f)\, D^k f(v,z)}{M(v)}\, dvd\mu}
_{\text{RHS}}
\end{equation}
where 
\begin{align*}&\displaystyle\text{RHS} = \int_{\mathcal S}\int_{\mathcal T}\, \sum_{j=0}^k \binom{k}{j} v^2 D^{k-j}\sigma\, D^j f(w,z) D^k f(v,z)\, dwdvd\mu \\[4pt]
&\displaystyle \qquad\quad-\int_{\mathcal S}\int_{\mathcal T}\, \sum_{j=0}^k \binom{k}{j} v^2 D^{k-j}\sigma M(w)\frac{D^j f(v,z) D^k f(v,z)}{M(v)}\, dwdvd\mu \\[4pt]
&\displaystyle \qquad= \underbrace{\int_{\mathcal S}\int_{\mathcal T}\, v^2 \sigma D^k f(w,z) D^k f(v,z)\, dwdvd\mu}_{\textcircled{c}}  
\underbrace{- \int_{\mathcal S}\int_{\mathcal T} \sigma M(w) \frac{(v D^k f(v,z))^2}{M(v)}\, dwdvd\mu}_{\textcircled{d}} \\[4pt]
&\displaystyle \qquad\quad+\underbrace{\int_{\mathcal S}\int_{\mathcal T}\,  \sum_{j=0}^{k-1} \binom{k}{j} v^2 D^{k-j}\sigma\, D^j f(w,z) D^k f(v,z)\, dwdvd\mu}_{\textcircled{e}} \\[4pt]
&\displaystyle \qquad\quad\underbrace{-\int_{\mathcal S}\int_{\mathcal T}\, \sum_{j=0}^{k-1} \binom{k}{j} v^2 D^{k-j}\sigma M(w)\frac{D^j f(v,z) D^k f(v,z)}{M(v)}\, dwdvd\mu}_{\textcircled{f}}\,.
\end{align*}
By arguments similar to (\ref{aa}), (\ref{bb}) and the uniform regularity of $f$ given by Theorem \ref{thm1}, 
one directly has $$ \textcircled{c} \leq \widehat C_{\sigma} ||D^k (vf)||_{\Gamma}\, ||D^k f||_{\Gamma}\leq C \widehat C_{\sigma}\, ||D^k (vf)||_{\Gamma}\,, \qquad
\textcircled{d} \leq -\sigma_{\rm min} \, ||D^k (vf)||_{\Gamma}^2\,.$$
Now we estimate $\textcircled{e}$ and $\textcircled{f}$:
\begin{align*}
&\displaystyle \textcircled{e} \leq \sum_{j=0}^{k-1}\binom{k}{j} \int_{\mathcal S}\left(\int_{\mathcal T} v^2 (D^{k-j}\sigma)^2 M(v)M(w)\, dwdv \right)^{\frac{1}{2}}
\\[4pt]
&\displaystyle\qquad\cdot\left(\int_{\mathcal T} \frac{(D^j f(w,z))^2}{M(w)}\, \frac{(v D^k f(v,z))^2}{M(v)}\, dwdv\right)^{\frac{1}{2}}d\mu
\leq\sum_{j=0}^{k-1}\binom{k}{j} \widehat C_{\sigma}\, ||D^j f||_{\Gamma}\, ||D^k (vf)||_{\Gamma}\\[4pt]
&\displaystyle\quad\leq C \widehat C_{\sigma} \sum_{j=0}^{k-1}\binom{k}{j}\, ||D^k (vf)||_{\Gamma}\leq 2^k C \widehat C_{\sigma}\, ||D^k (vf)||_{\Gamma}\,,
\end{align*}
where we used Theorem \ref{thm1} in the last inequality. By Young's inequality, 
\begin{align}\label{f}&\displaystyle \textcircled{f}  \leq \frac{\sigma_{\rm min}}{2}||D^k (vf)||_{\Gamma}^2 + \frac{1}{2\sigma_{\rm min}}
\bigg|\bigg|\, \sum_{j=0}^{k-1}\binom{k}{j} \int_{\mathbb R^d}\, (D^{k-j}\sigma) M(w)dw\, D^j (vf)\, \bigg|\bigg|_{\Gamma}^2\,,
\end{align}
then using the Cauchy Schwartz inequality, 
\begin{align*}
&\displaystyle\quad
\bigg|\bigg|\, \sum_{j=0}^{k-1}\binom{k}{j} \int_{\mathbb R^d}\, (D^{k-j}\sigma) M(w)dw\, D^j (vf)\, \bigg|\bigg|_{\Gamma}^2 \\[4pt]
&\displaystyle \leq 
\left(\, \sum_{j=0}^{k-1}\binom{k}{j}^2 \max_{0\leq j\leq k}\bigg|\bigg|\int_{\mathbb R^d}\, (D^j \sigma) M(w)dw\, \bigg|\bigg|_{L^{\infty}}^2 \right)\left(\, \sum_{j=0}^{k-1} ||D^j (vf)||_{\Gamma}^2
\right)\leq 4^k \lambda_1^2\, ||vf||_{H^{k-1}}^2.
\end{align*}
Thus (\ref{f}) gives $$\textcircled{f}\leq\frac{\sigma_{\rm min}}{2}||D^k (vf)||_{\Gamma}^2 + \frac{4^k\lambda_1^2}{2\sigma_{\rm min}}\, ||vf||_{H^{k-1}}^2\,.$$
Sum up $\textcircled{c}, \textcircled{d}, \textcircled{e}, \textcircled{f}$, one gets
\begin{equation}\label{keqn}\frac{\epsilon^2}{2}\partial_t ||D^k (vf)||_{\Gamma}^2 \leq  -\frac{\sigma_{\rm min}}{2}||D^k (vf)||_{\Gamma}^2 +
\widetilde C\, ||D^k (vf)||_{\Gamma}  +  \frac{4^{k}\lambda_1^2}{2\sigma_{\rm min}}\, ||vf||_{H^{k-1}}^2\,,\end{equation}
where $\widetilde C= (2^k+1) C \widehat C_{\sigma}$. 

Assume that for any $k\leq p$, where $p\in\mathbb N$, the conclusion (\ref{MI2}) holds. Adding these inequalities together, 
\begin{align}\label{k1}&\displaystyle\quad
\epsilon^2 \partial_t \bigg(\, \sum_{i=0}^p ||D^i (vf)||_{\Gamma}^2 + 
\sum_{i=1}^p \sum_{j=0}^{i-1}c_{ij} || D^j (vf)||_{\Gamma}^2\, \bigg)\notag\\[4pt]
&\displaystyle\leq-\sigma_{\rm min} \sum_{i=1}^{p}\sum_{j=0}^{i}||D^j (vf)||_{\Gamma}^2 + \sum_{i=1}^{p}\sum_{j=0}^{i}s_{ij}||D^j (vf)||_{\Gamma}
-2\sigma_{\rm min}||vf||_{\Gamma}^2 + 2\widehat C_{\sigma}C\, ||vf||_{\Gamma}, 
\end{align}
which is equivalent to
\begin{equation}\label{p}\epsilon^2 \partial_t \bigg(\, \sum_{j=0}^p c_{p+1,j}^{\prime} ||D^j f||_{\Gamma}^2\, \bigg)\leq 
- \sum_{j=0}^p \gamma_{p+1,j} ||D^j (vf)||_{\Gamma}^2 +\sum_{j=0}^p s_{p+1,j}^{\prime} ||D^j(vf)||_{\Gamma}\,,
\end{equation}
where 
\begin{align*}
c_{p+1,j}^{\prime} =
\begin{cases} 
                         1 + \sum_{i=1}^p c_{i j}, \qquad 0\leq j\leq p-1, \\[4pt]       
                        1, \qquad\qquad\qquad j=p,           
\end{cases}
\end{align*}
\begin{minipage}{.5\linewidth}
\begin{align*}
\displaystyle
\gamma_{p+1,j} =
\begin{cases}   \left(2+\frac{(p+1)p}{2}\right)\sigma_{\rm min}, \qquad j=0, \\[4pt]
                         \frac{(p+1)p}{2}\, \sigma_{\rm min}, \qquad\quad1\leq j\leq p,         
\end{cases}
\end{align*}
\end{minipage} 
\begin{minipage}{.5\linewidth}
\begin{align*}
\displaystyle
s_{p+1,j}^{\prime} =
\begin{cases}    2\widehat C_{\sigma}C +\sum_{i=1}^p s_{i0}, \quad j=0, \\[4pt]
                         \sum_{i=1}^p s_{ij}, \qquad\quad 1\leq j\leq p. 
\end{cases}
\end{align*}
\end{minipage} 
\\[20pt]

When $k=p+1$, (\ref{keqn}) gives 
\begin{equation}\label{p1} \epsilon^2\partial_t ||D^{p+1} (vf)||_{\Gamma}^2 \leq  -\sigma_{\rm min}\, ||D^{p+1} (vf)||_{\Gamma}^2 +
2\widetilde C\, ||D^{p+1} (vf)||_{\Gamma}  +  \frac{4^{p+1}\lambda_1^2}{\sigma_{\rm min}}\, ||vf||_{H^p}^2\,. 
\end{equation}
\\[2pt]
Multiplying (\ref{p}) by $\chi_{p+1}$(which is positive and defined below) and adding with (\ref{p1}) gives
\begin{align}
&\displaystyle\epsilon^2 \partial_t \bigg( ||D^{p+1}(vf)||_{\Gamma}^2
 + \sum_{j=0}^p c_{p+1,j}^{\prime\prime}\, ||D^j (vf)||_{\Gamma}^2 \bigg) \leq 
-\sigma_{\rm min}\, \sum_{j=0}^{p+1} ||D^j (vf)||_{\Gamma}^2 \notag\\[4pt]
&\label{p2}\displaystyle \qquad\qquad\qquad\qquad\qquad\qquad\qquad\qquad\qquad\qquad\quad
 +\sum_{j=0}^{p+1} s_{p+1,j}^{\prime\prime}\, ||D^j (vf)||_{\Gamma}\,,
\end{align}
where \\[4pt]
\begin{minipage}{.28\linewidth}
\begin{align*}\chi_{p+1}= \frac{\frac{4^{p+1}\lambda_1^2}{\sigma_{\rm min}}+\sigma_{\rm min}}{\min_{0\leq j\leq p}\{\gamma_{p+1,j}\}}\,,
\end{align*}
\end{minipage}
\begin{minipage}{.28\linewidth}
$$ c_{p+1,j}^{\prime\prime} = c_{p+1,j}^{\prime}\,\chi_{p+1}\,,$$
\end{minipage}
\begin{minipage}{.28\linewidth}
\begin{align*}
\displaystyle
s_{p+1,j}^{\prime\prime} =
\begin{cases}
2\widetilde C,\qquad\qquad j=p+1,  \\[4pt]
s_{p+1,j}^{\prime}\,\chi_{p+1}, \,  0\leq j\leq p\,. 
\end{cases}
\end{align*}
\end{minipage}
This shows that (\ref{MI2}) still holds for $k=p+1$. By Mathematical Induction, conclusion (\ref{MI2}) holds for all $k\in\mathbb N$. Thus 
we finish the proof of Lemma \ref{L2}. 
\end{proof}
The following theorem provides a new estimate on the regularity of $v\cdot\nabla_x f$, which is of exponential decay. 
\begin{theorem}
\label{thm3}
If assumptions for the collision kernel, namely (\ref{sigma}), (\ref{assump1}) and (\ref{assump3}) are satisfied  
and if for some integer $m\geq 0$, $$||D^k (vf_0)||_{\Gamma}\leq C_0, \qquad
||D^k (v\cdot\nabla_x f_0)||_{\Gamma}\leq C_1, \qquad k=0, \cdots, m, $$
then the following regularity results of $v f$ and $v\cdot\nabla_x f$ hold:
$$||D^k (vf)||_{\Gamma}\leq \widetilde C_0 (e^{-\frac{C_2 t}{\epsilon^2}}+1),  \qquad k=0, \cdots, m, $$
and
$$||D^k (v\cdot\nabla_x f)||_{\Gamma}\leq \widetilde C_1 (e^{-\frac{C_2  t}{\epsilon^2}}+1), \qquad k=0, \cdots, m,$$
where $C_0$, $C_1$, $\widetilde C_0$, $\widetilde C_1$ and $C_2$ are constants independent of $\epsilon$. 
\end{theorem}

\begin{proof}
Define the weighted energy norm:
$$||vf||_{\widehat \Gamma^k}^2: = \sum_{j=0}^{k} \widetilde c_{k+1,j}^{\prime\prime}\, ||D^j (vf)||_{\Gamma}^2\,,$$
where 
\begin{align*}
\widetilde c_{k+1,j}^{\prime\prime}=
\begin{cases}
1, \qquad\quad j=k, \\[4pt]
c_{k,j}^{\prime\prime}, \qquad 0\leq j\leq k-1. 
\end{cases}
\end{align*}

The second term in (\ref{MI2}) has the estimate 
\begin{equation}\label{L2a}\sum_{j=0}^{k} s_{kj} ||D^j (vf)||_{\Gamma} \leq \bigg(\sum_{j=0}^k \frac{(s_{kj})^2}{ \widetilde c_{k+1,j}^{\prime\prime}}\bigg)^{\frac{1}{2}}\bigg(\sum_{j=0}^k \widetilde c_{k+1,j}^{\prime\prime}\, ||D^j(vf)||_{\Gamma}^2\bigg)^{\frac{1}{2}}=\widetilde C_k\, ||vf||_{\widehat \Gamma^k}\,,
\end{equation}
where the Cauchy Schwartz inequality is used, 
and the constant $\displaystyle \widetilde C_k^2 =\sum_{j=0}^k \frac{(s_{kj})^2}{\widetilde c_{k+1,j}^{\prime\prime}}$. 
The first term in (\ref{p2}) is estimated by
\begin{equation}\label{L2b} -\sigma_{\rm min}\, \sum_{j=0}^k ||D^j (vf)||_{\Gamma}^2 \leq -\frac{\sigma_{\rm min}}{\max_{0\leq j\leq k}\{\widetilde c_{k+1,j}^{\prime\prime}\}}
\, ||vf||_{\widehat \Gamma^k}^2\,.
\end{equation}
Therefore, according to Lemma \ref{L2} and (\ref{L2a}), (\ref{L2b}), 
$$\epsilon^2 \partial_t ||vf||_{\widehat \Gamma^k}^2 \leq -\widetilde \sigma_{\rm min} ||vf||_{\widehat \Gamma^k}^2 +\widetilde C_k\, ||vf||_{\widehat \Gamma^k}\,,$$
where $\displaystyle\widetilde\sigma_{\rm min} =\frac{\sigma_{\rm min}}{\max_{0\leq j\leq k}\{\widetilde c_{k+1,j}^{\prime\prime}\}}$. 
Cancel $||vf||_{\widehat \Gamma^k}$ on both sides and use Gronwall's inequality, 
$$ ||D^j (vf)||_{\Gamma}\lesssim ||vf||_{\widehat \Gamma^k}\leq  \widetilde C_0 (e^{-\frac{C_2 t}{\epsilon^2}}+1), $$
for $j=0, \cdots, k$, where $C_2=\widetilde\sigma_{\rm min}/2$ and $\widetilde C_0$ are constants independent of $\epsilon$. 

Notice that $\nabla_x f$ and $f$ satisfy the same equation, under the assumptions given in Theorem \ref{thm3}, one consequently has
$$||D^k (v\cdot\nabla_x f)||_{\Gamma} \leq \widetilde C_1 (e^{-\frac{C_2 t}{\epsilon^2}}+1), $$
where $C_1$, $C_2$ are independent of $\epsilon$. This completes the proof.  \\[2pt]
\end{proof}

\subsection{Regularity of $\Pi f-f$}
\label{Pi_f}

Our goal of this subsection is to obtain a regularity estimate on $\Pi f-f$, as shown in Theorem \ref{thm2} below. 
\begin{theorem} {\bf({Estimate on $\Pi f-f$})}
\label{thm2}
If all the assumptions in Theorem \ref{thm1} and Theorem \ref{thm3} are satisfied,  
then the regularity of $\Pi f -f$ is given by 
\begin{align}
&\label{thm2a}\displaystyle ||D^k (\Pi f -f)||_{\Gamma}^2  \leq e^{-\sigma_{\rm min}t/2\epsilon^2}\, ||D^k (\Pi f_0 - f_0)||_{\Gamma}^2 + C^{\prime}\epsilon^2 
\leq C^{\prime\prime}\epsilon^2, 
\end{align}
for any $t\in (0, T]$ and $0\leq k\leq m$, where $C^{\prime}$ and $C^{\prime\prime}$ are constants independent of $\epsilon$. 
 \end{theorem}  

\begin{proof}
Take the projection $\Pi$ on both sides of (\ref{eqn}), 
\begin{equation}\label{PI}\epsilon^2 \partial_t (\Pi f) + \epsilon \Pi (v\cdot\nabla_x f) = 0. \end{equation}
Subtract (\ref{PI}) by (\ref{eqn}), 
\begin{equation}\label{PI1} \epsilon^2 \partial_t (\Pi f -f) + \epsilon \left( \Pi(v\cdot \nabla_x f) - v\cdot\nabla_x f\right) = -\mathcal Q(f). \end{equation}
Differentiating (\ref{PI1}) $k$ times and taking the scalar product with $D^k (\Pi f-f)$, one gets
\begin{align}
&\displaystyle \epsilon^2 \partial_t ||D^k (\Pi f-f)||_{\Gamma}^2 = -2 \,\epsilon \underbrace{\int_{\mathcal S}\,
\langle D^k \Pi (v\cdot \nabla_x f) - D^k (v\cdot\nabla_x f), D^k (\Pi f-f) \rangle_{\widetilde L^2}\, d\mu}_{I}  \notag\\[4pt]
&\label{II}\displaystyle\qquad\qquad\qquad\qquad\qquad +
2\underbrace{\int_{\mathcal S}\,\langle D^k \mathcal Q(f), D^k (f-\Pi f)\rangle_{\widetilde L^2}\, d\mu}_{II}\,.
\end{align}
Notice that $\displaystyle\int_{\mathcal S}\,\langle D^k \mathcal Q(f), D^k \Pi f\rangle_{\widetilde L^2}^2\, d\mu = 0$, 
the estimate of term $II$ is given in (\ref{RHS1}). 

To estimate term $I$, since $$\int_{\mathcal S}\,\langle D^k \Pi (v\cdot \nabla_x f), D^k (\Pi f-f) \rangle_{\widetilde L^2}\, d\mu=0, $$
it remains to estimate $\displaystyle \epsilon\int_{\mathcal S}\,\langle D^k (v\cdot\nabla_x f), D^k (\Pi f-f) \rangle_{\widetilde L^2}\, d\mu$. 
In section \ref{vf}, the 
 estimate for $||D^k (v\cdot\nabla_x f)||_{\Gamma}$ has been done. 
With the help of Theorem \ref{thm3}, the rest of the proof mostly follows
\cite{JinLiuMa-2016}. For completeness, we write it out. 

By Young's inequality, 
\begin{align}
&\displaystyle\quad\epsilon\int_{\mathcal S}\langle D^k (v\cdot\nabla_x f), \, D^k (\Pi f-f) \rangle_{\widetilde L^2}\, d\mu \notag\\[4pt]
&\displaystyle \leq \frac{\sigma_{\rm min}}{4}\, ||D^k (\Pi f -f)||_{\Gamma}^2 +\frac{\epsilon^2}{\sigma_{\rm min}}\, ||D^k (v\cdot\nabla_x f)||_{\Gamma}^2 \notag\\[4pt]
&\label{I}\displaystyle \leq \frac{\sigma_{\rm min}}{4}\, ||D^k (\Pi f -f)||_{\Gamma}^2  + \frac{\widetilde C_1 \epsilon^2}{\sigma_{\rm min}} (e^{-\frac{C_2 t}{\epsilon^2}}+1)  \leq \frac{\sigma_{\rm min}}{4}\, ||D^k (\Pi f -f)||_{\Gamma}^2  + \frac{2 \widetilde C_1 \epsilon^2}{\sigma_{\rm min}}\,.
\end{align}
By (\ref{II}), using the estimate (\ref{RHS1}) and (\ref{I}), one gets
\begin{equation}\label{MI3}\epsilon^2\partial_t ||D^k (\Pi f -f)||_{\Gamma}^2 
\leq -\frac{\sigma_{\rm min}}{2} ||D^k(\Pi f -f)||_{\Gamma}^2 + \frac{C_{\sigma}^2 4^k}{\sigma_{\rm min}}
\, ||\Pi f -f ||_{\Gamma^{k-1}}^2 + \frac{4 \widetilde C_1 \epsilon^2}{\sigma_{\rm min}}. \end{equation} \\[0.1pt]

We prove Theorem \ref{thm2} using Mathematical Induction. When $k=0$, (\ref{MI3}) becomes 
$$\epsilon^2 \partial_t ||\Pi f-f||_{\Gamma}^2 \leq -\frac{\sigma_{\rm min}}{2} ||\Pi f -f||_{\Gamma}^2 + \frac{4 \widetilde C_1 \epsilon^2}{\sigma_{\rm min}}. $$
By Gronwall's inequality,  $$||\Pi f-f||_{\Gamma}^2 \leq e^{-\sigma_{\rm min} t/2\epsilon^2}\, ||\Pi f_0 -f_0||_{\Gamma}^2 +
\frac{8 \widetilde C_1}{\sigma_{\rm min}^2}
\epsilon^2 \leq C_0 \epsilon^2, $$
which satisfies (\ref{thm2a}).  
Assume for any $k\leq p$ where $p\in\mathbb N$, the conclusion (\ref{thm2a}) holds. Thus 
$$ ||\Pi f -f||_{\Gamma^p}^2 \leq C_p\,\epsilon^2. $$

When $k=p+1$, (\ref{MI3}) reads
$$\epsilon^2\partial_t ||D^{p+1} (\Pi f -f)||_{\Gamma}^2 
\leq -\frac{\sigma_{\rm min}}{2} ||D^{p+1} (\Pi f-f)||_{\Gamma}^2 + \frac{C_{\sigma}^2\, 4^{p+1}}{\sigma_{\rm min}} C_p\, \epsilon^2 + 
\frac{4 \widetilde C_1 \epsilon^2}{\sigma_{\rm min}},$$
which is equivalent to 
$$\partial_t ||D^{p+1} (\Pi f -f)||_{\Gamma}^2 \leq -\frac{\sigma_{\rm min}}{2\epsilon^2} ||D^{p+1} (\Pi f-f)||_{\Gamma}^2 + C_{p+1}^{\prime}, $$
where $\displaystyle C_{p+1}^{\prime}=\frac{C_{\sigma}^2\, 4^{p+1}}{\sigma_{\rm min}} C_p + 
\frac{4 \widetilde C_1}{\sigma_{\rm min}}$. 
By Gronwall's inequality, 
$$||D^{p+1} (\Pi f -f)||_{\Gamma}^2 \leq e^{-\sigma_{\rm min} t/2\epsilon^2}\, ||D^{p+1} (\Pi f -f)||_{\Gamma}^2 + C_{p+1}^{\prime\prime}\epsilon^2 \leq
C_{p+1}\epsilon^2, $$
where $C_{p+1}^{\prime\prime}$ and $C_{p+1}$ are constants independent of $\epsilon$. By Mathematical Induction, we finish the proof of Theorem \ref{thm2}.    
\end{proof}

\section{A Uniform Spectral Convergence in $\epsilon$}
\label{UC}
The main purpose of this section is to obtain the uniform spectral convergence of the gPC-SG method for problem (\ref{eqn}), 
as shown in Theorem \ref{thm4}. 

Let $f$ be the solution to (\ref{eqn}). We define the $K$-th order projection operator 
$$P_{K}f=\sum_{|\bj|=1}^{K}\langle f,\psi_{\bj}\rangle_{\pi}\, \psi_{\bj}\,.$$
The error arisen from the gPC-SG can be split into two parts $R_{K}$ and $e_{K}$,
\begin{equation}f-f_{K}=(f-P_{K}f) + (P_{K}f-f_{K}):=R_{K}+e_{K},\label{f_K}\end{equation}
where $R_K=f-P_{K}f$ is the projection error, and 
$$e_K=P_{K}f-f_{K}=\sum_{|\bj|=1}^{K}\big(\langle f,\psi_{\bj}\rangle_{\pi}-f_{\bj}\big)\, \psi_{\bj}=\hat\be\cdot\boldsymbol{\psi},$$
where $\displaystyle\hat\be=\big(\langle f,\psi_1\rangle_{\pi}-f_1, \cdots, \langle f,\psi_{K}\rangle_{\pi}-f_{K}\big)$ is the numerical error, and $\boldsymbol{\psi}=(\psi_1,\cdots,\psi_{K})$. 
\\[2pt]

Define the operator $\displaystyle\mathcal L = \epsilon^2 \partial_t +\epsilon v\cdot\nabla_x -\mathcal Q$. 
Recall the Lemma given in \cite{JinLiu-2016}: 
\begin{lemma}
\label{lemma1}
The operator  
$\displaystyle\mathcal L = \epsilon^2 \partial_t +\epsilon v\cdot\nabla_x -\mathcal Q$ satisfies the following equality:
$$\langle\mathcal L(R_K),\boldsymbol{\psi}\rangle_{\pi}=-\langle\mathcal Q(R_K),\boldsymbol{\psi}\rangle_{\pi}. $$ \\[1pt]
\end{lemma}

Now since $\mathcal L(f)=0$, $P_{K}\mathcal L(f_K)=0$, thus 
\begin{equation} \label{error1}
\langle\mathcal{L}(e_K), \boldsymbol\psi\rangle_{\pi}=-\langle\mathcal{L}(R_K), \boldsymbol\psi\rangle_{\pi}\,. \end{equation}
Denote $\displaystyle d\chi=\frac{1}{M(v)}\, dvdx\pi(z)dz$, and $\mathcal W=\Omega\times\mathbb R^d\times I_z$. 
Taking the scalar product of (\ref{error1}) with $\hat\be$ and integrating on $\mathcal W$, one gets
$$\frac{\epsilon^2}{2}\partial_t ||e_K||_{\Gamma}^2 - \int_{\mathcal W}\, \langle \mathcal Q(e_K), \boldsymbol{\psi}\rangle_{\pi} \cdot\hat\be\, d\chi
= -\int_{\mathcal W}\, \langle \mathcal L(R_K), \boldsymbol{\psi}\rangle_{\pi} \cdot\hat\be\, d\chi\,,$$
that is, $$\frac{\epsilon^2}{2}\partial_t ||e_K||_{\Gamma}^2 =\underbrace{\int_{\mathcal W}\, \langle \mathcal Q(e_K), e_K \rangle_{\pi}\, d\chi}_{III}
\underbrace{-\int_{\mathcal W}\, \langle \mathcal L(R_K), e_K \rangle_{\pi}\, d\chi}_{IV}\,.$$
Notice that $\displaystyle\int_{\mathcal W}\, \langle \mathcal Q(R_K), \,\Pi e_K \rangle_{\pi}\, d\chi =0$, thus
\begin{equation}\label{III} \int_{\mathcal W}\, \langle \mathcal Q(R_K), e_K \rangle_{\pi}\, d\chi = \int_{\mathcal W}\, \langle \mathcal Q(R_K),\, e_K - \Pi e_K \rangle_{\pi}\, d\chi\,.
\end{equation}
Since $||\sigma||_{L^{\infty}(v,z)}\leq C_{\sigma}$, then 
\begin{align*}
&\displaystyle \left|\mathcal Q(R_K)\right| \leq C_{\sigma}\left| M(v)\int_{\mathbb R^d} R_K(w,z)dw - \int_{\mathbb R^d}M(w)dw\, R_K(v,z)\right| \\[2pt]
&\displaystyle\qquad\qquad  =C_{\sigma}\left|\Pi R_K -R_K\right|, 
\end{align*}
which gives 
\begin{equation}\label{IV}||\mathcal Q(R_K)||_{\Gamma}^2 \leq C_{\sigma}^2 \, ||R_K-\Pi R_K||_{\Gamma}^2\,.\end{equation}
According to Lemma \ref{lemma1}, (\ref{III}), (\ref{IV}) and Young's inequality, one has 
\begin{align}
&\displaystyle IV = -\int_{\mathcal W}\, \langle \mathcal L(R_K), e_K\rangle_{\pi}\, d\chi = 
\int_{\mathcal W}\, \langle \mathcal Q(R_K), \,e_K-\Pi e_K \rangle_{\pi}\, d\chi \notag\\[4pt]
&\displaystyle\quad\leq \frac{\sigma_{\rm min}}{2}|| e_K -\Pi e_K ||_{\Gamma}^2 + \frac{1}{2\sigma_{\rm min}} ||\mathcal Q(R_K)||_{\Gamma}^2 \notag
\\[4pt]
&\label{error2}\displaystyle\quad\leq  \frac{\sigma_{\rm min}}{2}|| e_K -\Pi e_K ||_{\Gamma}^2 + \frac{C_{\sigma}^2}{2\sigma_{\rm min}} ||R_K-\Pi R_K||_{\Gamma}^2\,.
\end{align}

By the standard error estimate for orthogonal polynomial approximations and Theorem \ref{thm1}, 
\begin{equation} \label{RK1} ||R_K||_{\Gamma} \leq C_1 K^{-m}\, ||D^m f||_{\Gamma}\leq\frac{C_1 C}{K^m}\,.
\end{equation}
According to Theorem \ref{thm2}, 
\begin{align}&\displaystyle ||R_K-\Pi R_K||_{\Gamma} = || (\Pi f -f) - (\Pi (P_K f) - P_K f) ||_{\Gamma}\notag \\[4pt]
&\label{RK2}\displaystyle  \qquad\qquad\qquad\quad \leq C_2^{\prime} K^{-m}\, ||D^m (\Pi f -f)||_{\Gamma} 
\leq \frac{C_2}{K^m}\epsilon\,,
\end{align}
where $C_2=C_2^{\prime}C^{\prime\prime}$. 
By the coercivity property of $\mathcal Q$, 
\begin{equation}\label{eK1}III=\int_{\mathcal W}\, \langle \mathcal Q(e_K), \, e_K \rangle_{\pi}\, d\chi \leq -\sigma_{\rm min}||e_K -\Pi e_K||_{\Gamma}^2\,.
\end{equation}
Adding up terms $III$ and $IV$, using (\ref{error2}), (\ref{RK2}) and (\ref{eK1}), one has
\begin{align*}&\displaystyle \frac{\epsilon^2}{2}\partial_t ||e_K||_{\Gamma}^2  \leq -\frac{\sigma_{\rm min}}{2} ||e_K -\Pi e_K||_{\Gamma}^2 + \frac{C_{\sigma}^2}{2\sigma_{\rm min}}
 ||R_K-\Pi R_K||_{\Gamma}^2 \\[4pt]
&\displaystyle\qquad\qquad\quad\leq -\frac{\sigma_{\rm min}}{2} ||e_K -\Pi e_K||_{\Gamma}^2 + \frac{1}{2\sigma_{\rm min}}\left(\frac{C_{\sigma} C_2}{K^m}\right)^2\epsilon^2  \leq \left(\frac{C}{K^m}\right)^2 \epsilon^2\,. 
 \end{align*}
Thus, \begin{equation}\label{eK2}||e_K||_{\Gamma} \leq \frac{C(T)}{K^m}\,.\end{equation}

Now we can conclude the following theorem on the uniform convergence in $\epsilon$ of the stochastic Galerkin method. 
\begin{theorem}
\label{thm4}
If all the assumptions in Theorem \ref{thm1}, Theorem \ref{thm3} and Theorem \ref{thm2} are satisfied, 
the error of the gPC-SG method is given by 
$$ ||f-f_K||_{\Gamma} \leq \frac{C(T)}{K^m}\,,$$
where $C(T)$ is a constant independent of $\epsilon$. 
\end{theorem}
\begin{proof}
Using (\ref{RK1}) and (\ref{eK2}), one has
$$ ||f-f_K||_{\Gamma} \leq ||R_K||_{\Gamma} + ||e_K||_{\Gamma} \leq \frac{C(T)}{K^m}\,,$$
where $C(T)$ is a constant independent of $\epsilon$. This completes the proof. 
\end{proof}

\begin{remark}
Differed from \cite{JinLiuMa-2016} for this part, one additionally needs Theorem \ref{thm3} to 
complete the proof of uniform spectral convergence of the SG method for numerically solving problem (\ref{eqn}). 
\end{remark}

\section{Conclusion}
\label{Con}
In this paper, we establish the {\it uniform-in-Knudsen-number} spectral accuracy of the stochastic Galerkin method for 
the linear semiconductor Boltzmann equation with random inputs and diffusive scalings, 
which consequently allows us to justify the stochastic AP property of the gPC-based stochastic Galerkin method proposed in \cite{JinLiu-2016}. 
Extensive numerical examples have been shown in \cite{JinLiu-2016} to validate the main result of this paper: 
{\it uniform spectral convergence} of the gPC-SG method, i.e., 
the number of polynomial chaos can be chosen {\it independent} of the Knudsen number, yet can still capture the solutions to the 
Galerkin system of the limiting drift-diffusion equations shown in (\ref{diff2}), with a spectral accuracy. 
It is expected that our approach to prove the uniform convergence of the stochastic Galerkin method will be useful for 
more general kinetic equations, for example when the external potential is involved. 

\section*{Acknowledgement}
The author would like to thank Prof. Shi Jin and Prof. Jianguo Liu for encouraging the author to think about this project. 


\medskip
\medskip

\bibliographystyle{siam}
\bibliography{main.bib}

\end{document}